\documentclass[11pt]{amsart}
\usepackage{amsthm}
\usepackage{amsmath}
\usepackage[all]{xy}
\usepackage{amsfonts}
\usepackage{amssymb}
\usepackage{amscd}
\usepackage{comment}

\newtheorem{theorem}{Theorem}
\newtheorem{proposition}{Proposition}

\newtheorem{lemma}{Lemma}

\newcommand{\Z}{{\mathbb Z}}
\newcommand{\Q}{{\mathbb Q}}

\DeclareMathOperator{\sgn}{sgn}

\author[D.\,L. Pincus]{David L. Pincus}
\address{Department of Mathematics, University of Maryland, College Park, MD 20742, USA}
\email{pincus@umd.edu}
\author[L.\,C. Washington]{Lawrence C. Washington}
\address{Department of Mathematics, University of Maryland, College Park, MD 20742, USA}
\email{lcw@umd.edu}
\date{\today}
\title[On The Field Isomorphism Problem]{On the Field Isomorphism Problem for the Family of Simplest Quartic Fields}
\subjclass[2020]{11R16,\,11B37}

\begin{document}
\maketitle
\begin{abstract}
Deciding whether or not two polynomials have isomoprhic splitting fields over the rationals is 
the Field Isomorphism Problem.  We consider polynomials of the 
form \mbox{$f_n(x) = x^4-nx^3-6x^2+nx+1$} with $n \neq 3$ a positive integer 
and we let $K_n$ denote the splitting field of $f_n(x)$; a `simplest quartic field'.  
Our main theorem states that under certain hypotheses there can be at most one positive integer $m \neq n$ such that $K_m=K_n$.  The proof relies on the existence of squares in recurrent sequences and a result of J.H.E. Cohn \cite{cohn1}.  
These sequences allow us to establish uniqueness of the splitting field under additional hypotheses in Section~\ref{S:5} 
and to establish a connection with elliptic curves in Section~\ref{S:6}.
\end{abstract}

\section{Introduction}\label{S:1}

Let $n\in \mathbb Z\setminus \{0,\pm 3\}$ and let $K_n=\mathbb Q(\rho)$, where $\rho$ is a root of 
$$
f_n(x)= x^4-nx^3-6x^2+nx+1.
$$
Since $K_{-n}=K_n$, we henceforth restrict $n$ to be a positive integer.
These fields $K_n$ are the {\it simplest quartic fields} of M.-N. Gras \cite{gras2}.  Their Galois groups are cyclic of order 4, given by the identity map and
$$
\sigma:  \rho\mapsto \frac{\rho-1}{\rho+1}, \qquad \sigma^2:  \rho\mapsto \frac{-1}{\rho}, \qquad \sigma^3: \rho\mapsto \frac{1+\rho}{1-\rho}.
$$

It is natural to ask if $K_m=K_n$ for distinct positive integers $m$ and $n$. 
Hoshi \cite{hoshi1} gives the following examples:
$$
K_1=K_{103}, \qquad K_2 = K_{22}, \qquad K_4 = K_{956},
$$
and proves that these are the only ones with $0<m<n\le 10^5$, and also these are the only ones with $0<m\le 10^3$ with no bound on $n$.

In the following, we take an approach different from the one used by Hoshi and prove the following:
\begin{theorem}\label{mainthm}(a) Let $n\equiv 2\pmod 4$ be a positive integer. There is at most one positive integer $m\ne n$ such that $K_m=K_n$.\newline
 (b) Let $n\equiv 8 \pmod {16}$ be a positive integer and suppose the trace to $\mathbb Q$ of the fundamental unit of $\mathbb Q(\sqrt{n^2+16})$ is odd. There is at most one positive integer $m\ne n$ such that $K_m=K_n$.\newline
(c) Let $n$ be an arbitrary positive integer. There are only finitely many positive integers $m$ such that $K_m=K_n$.
\end{theorem}

{\bf Remarks.} Part (c) was also  proved by Hoshi, but using work of Siegel  in a form different from what we use. The congruence conditions on $m$ in parts (a) and (b)
are the only ones that allow the fundamental unit of $\mathbb Q(\sqrt{n^2+16})$ to have negative norm and odd trace, which we need in order to apply results of Cohn in Section 3.

\section{Preliminaries}\label{S:2}

Writing
$$
f_n(x)/x^2 = \left(x-\frac1x\right)^2 -n\left(x-\frac1x\right) -4,
$$
we see that the unique quadratic subfield of $K_n$ is $k_n=\mathbb Q\left(\sqrt{n^2+16}\right)$, and
$$
\rho-\frac1{\rho}= \frac{n\pm\sqrt{n^2+16}}{2}.
$$
It follows that
$$
K_n= k_n\left(\sqrt{2\left(n^2+16+n\sqrt{n^2+16}\right)}\right).
$$
If $k_n=\mathbb Q(\sqrt{d})$ with $d$ squarefree, then $dy^2=n^2+16$ for some positive integer $y$,  and $K_n=k_n\left(\sqrt{2(y^2d+ny\sqrt{d}})\right)$.

Suppose now that $K_m=K_n$. Then $k_m=k_n$, so
$$
m^2+16 = dx^2, \qquad n^2+16 = dy^2
$$
for some positive integers $x$ and $y$. Since 
$$
k_m\left(\sqrt{2(x^2d+mx\sqrt{d}})\right)=k_n\left(\sqrt{2(y^2d+ny\sqrt{d}})\right)
$$ 
and $k_m=k_n$, we obtain the following:

\begin{lemma}\label{mainlemma} With the notations above, $K_m=K_n$ if and only if
\begin{equation*}\label{mainequ}
\left(x^2d+mx\sqrt{d}\right)\left(y^2d+ny\sqrt{d}\right)
=
xyd\left(x\sqrt{d}+m\right)\left(y\sqrt{d}+n\right) \in \left(k_n^{\times}\right)^2.
\end{equation*}
\end{lemma}

{\bf Remark.} It is possible that $K_m\ne K_n$ even though $k_m=k_n$. For example, $k_{58}=k_2=\mathbb Q(\sqrt{5})$. Write $m^2+16=58^2+16 = 26^2\cdot 5$, so
$x=26$. Write $2^2+16=2^2\cdot 5$, so $y=2$. Then
\begin{align*}
xyd\left(x\sqrt{d}+m\right)\left(y\sqrt{d}+n\right) &= 52\cdot 5(26\sqrt{5}+58)(2\sqrt{5}+2) \\
 &= 13(8\sqrt{5})^2\left(\frac{\sqrt{5}+1}{2}\right)^8
\end{align*}
is not a square in $\mathbb Q(\sqrt{5})$ since $13$ is not a square  in this field. Lemma \ref{mainlemma} tells us that $K_{58}\ne K_2$. We could also deduce this from part (a) of 
Theorem \ref{mainthm} since we already know that $K_2=K_{22}$, and there cannot be a third index giving the same field.

\section{Proof of parts ($a$) and ($b$)}\label{S:3}

\begin{proof}\renewcommand{\qedsymbol}{}
In both part (a) and part (b), $n$ is even. Write $n=2n_1$. Since $d$ is taken to be squarefree,  $y=2y_1$ for an integer $y_1$.  Moreover, our congruence condition on $n$ implies
$d\equiv 5\pmod 8$. We have
$$
\frac{y\sqrt{d} + n}{4}=\frac{y_1\sqrt{d}+n_1}{2}
$$
is a unit in the ring of integers of $\mathbb Q(\sqrt{d})=\mathbb Q(\sqrt{n^2+16})$ whose norm to $\mathbb Q$ is $-1$.

Let $\epsilon$ be the fundamental unit of $k_n$. Then
$$
\frac{y_1\sqrt{d} + n_1}{2} = \epsilon^s
$$
for an odd integer $s$ since the norm is $-1$. Moreover, the norm of $\epsilon$ must also be $-1$.

When $n\equiv 2\pmod 4$, we have that $n_1$ is odd, so $\epsilon^s= (n_1+y_1\sqrt{d})/2\not\in \mathbb Z[\sqrt{d}]$. Therefore, $\epsilon\not\in \mathbb Z[\sqrt{d}]$ and  $t=\text{Trace}(\epsilon)$
is odd. 

When $n\equiv 8\pmod{16}$, we are assuming that $t$ is odd.

Note that 
$$
y_1= \frac{\epsilon^s-\overline{\epsilon}^s}{\sqrt{d}},
$$
where $\overline{\epsilon}$ denotes the Galois conjugate. 
The numbers $$
u_j=    \frac{\epsilon^j-\overline{\epsilon}^j}{\epsilon - \overline{\epsilon}}
$$
satisfy the relations 
$$
u_0=0, \quad u_1=1, \quad u_{j+2}=t u_{j+1} + u_j \text{ for } j\ge 0.
$$
We have
$$  \frac{\sqrt{d}}{\epsilon-\overline{\epsilon}} y_1=  u_s.
$$
\smallskip

Suppose there are integers $\ell$ and $m$ such that $K_{\ell}=K_m=K_n$. 
Since $m^2+16=dx^2$ and $\ell^2+16=dw^2$ for integers $x$ and $w$, and $d\equiv 5\pmod 8$, we must have $m\equiv 2\pmod 4$ or $m\equiv 8\pmod {16}$ and similarly for $\ell$. 

Write $m=2m_1$.  Then $x=2x_1$ for some integer $x_1$, and 
$$
\frac{x\sqrt{d}+m}{4}= \frac{x_1\sqrt{d}+m_1}{2}= \epsilon^r
$$
with $r$ odd. It follows, as with $y_1$, that 
$$
\frac{\sqrt{d}}{\epsilon-\overline{\epsilon}}x_1=u_r,
$$ 
where $u_r$ is a term in the above recurrent sequence.
Moreover, from Lemma (\ref{mainlemma}), 
$$
xyd\left(x\sqrt{d}+m\right)\left(y\sqrt{d}+n\right)
 = 64x_1y_1d\epsilon^{r+s}
\in \left( k_m\right)^2.
$$
Since $r+s$ is even and $d$ is a square in $k_m$,  we obtain that $x_1y_1$ is a square in $k_m$. This implies that
$$
u_r u_s = \left(\frac{\sqrt{d}}{\epsilon-\overline{\epsilon}}\right)^2 x_1y_1\in \left(k_m^{\times}\right)^2.
$$
Therefore, 
$$
u_r u_s\in \left(\mathbb Z\right)^2 \text{ or } u_r u_s\in d\left(\mathbb Z\right)^2.
$$

The above considerations also apply to $\ell$ and we obtain $u_q$ in the recurrent sequence with 
$$
u_r u_q\in \left(\mathbb Z\right)^2 \text{ or } u_r u_q\in d\left(\mathbb Z\right)^2.
$$

We now need the following result of J. H. E. Cohn \cite{cohn1}:

\begin{theorem} (Cohn) Let $a$ be an odd integer and define $P_n(a)$ by
$$
P_0(a)=0, \quad P_1(a) = 1, \quad P_{n+2}(a) = a\, P_{n+1}(a) + P_n(a) \text{ for } n\ge 0.
$$
The only solutions of the equation $y^2=P_r(a) P_s(a)$ with $0<r<s$ are (for arbitrary odd $b$) 
$$
(r, s,a)= (1, 2, b^2), (1, 12, 1),  (2, 12, 1), (3, 6, 1), (3, 6, 3).
$$
\end{theorem}

Since $t$ is odd, Cohn's result applies to our situation. No exceptional case has both $r$ and $s$ odd. Therefore,
we must have $u_ru_s\in d\left(\mathbb Z\right)^2$ and $u_ru_q\in d\left(\mathbb Z\right)^2$. This implies $u_su_q$ is the square of an integer. Again, this is impossible by Cohn's theorem.  This completes the proof of parts (a) and (b) of 
Theorem~\ref{mainthm}.
\end{proof}

The only known example with $m\equiv 2\pmod 4$ where two of these fields are equal is $K_2=K_{22}$. For $n=2$, we have $n^2+16=y^2 d=2^2 \cdot 5$, so we have $u_s=y_1=1$.
For $m=22$, we have $m^2+16= x^2 d = 10^2\cdot 5$, so $u_r=x_1=5$. Therefore, $u_r u_s=d$. This shows that it is necessary to consider the case where $u_r u_s$ is $d$ times a square. Moreover, Cohn's theorem shows that there is no other $m$ such that $K_m=K_2=K_{22}$. 

\section{Proof of part ($c$)}\label{S:4}

\begin{proof}
From Equation (\ref{mainequ}), we know that if $K_m=K_n$ then 
$$
xyd\left(x\sqrt{d}+m\right)\left(y\sqrt{d}+n\right)\in \left(k_m^{\times}\right)^2.
$$
Fix $n$, which also fixes $y$ and $d$. Note that the norm of $x\sqrt{d}+m$ is $-16$, so, modulo squares of units, 
there are only finitely many possibilities for 
$x\sqrt{d} + m$. Therefore, there are only finitely many possibilities modulo squares of units for 
$$yd\left(x\sqrt{d}+m\right)\left(y\sqrt{d}+n\right).$$
 It follows that
 there are algebraic integers $c_1, \dots, c_N$, depending only on $n$,  such that
$c_i x= g^2$ for some $i$ and some $g\in k_n^{\times}$. Note that $g$ must be an algebraic integer.
Therefore, 
$$
c_i^2(m^2+16)= c_i^2 x^2 d= dg^4
$$
for some $i$. A theorem of Siegel implies that each curve $c_i^2(Y^2+16)= dX^4$ has only finitely many solutions in the algebraic integers of $k_n$. Therefore, there are only finitely many $m$ such that $K_m=K_n$. \end{proof}

{\bf Remark.}  Aruna C and P Vanchinathan \cite{c2023exceptional} considered the family of polynomials $x^4-nx^3-x^2+nx+1$ and  showed that for fixed $n$, there are only finitely many $m$
that yield the same field. Their proof  used  the fact that if $\alpha$ is a root of this polynomial then $1-\alpha$ is also a unit of the ring of integers (that is,
$\alpha$ is an {\it exceptional unit}). A theorem of Siegel says that a number field has only finitely many exceptional units. This result of Siegel and the one we used are
different consequences of his work on $S$-unit equations.

\section{Recurrent sequences}\label{S:5}

Let $k_n$ be the quadratic subfield of $K_n$ and let $\epsilon$ be the fundamental unit of $k_n$. Let
$$
u_j=    \frac{\epsilon^j-\overline{\epsilon}^j}{\epsilon - \overline{\epsilon}},  \qquad v_j= \epsilon^j+\overline{\epsilon}^j.
$$
If $t=\text{Trace}(\epsilon)$, then $u_j$ and $v_j$ satisfy  the relations
\begin{align*}
 u_0=0, \quad u_1=1, \quad &u_{j+2}=t u_{j+1} + u_j , \\
 v_0=2, \quad v_1=t, \quad  &v_{j+2}=t v_{j+1} + v_j
\end{align*}
for $ j\ge 0$.

The proofs of the following relations are straightforward from the definitions of $u_j$ and $v_j$:
\begin{gather*}
2u_{r+s}=u_rv_s+u_sv_r\\
u_{2r}= u_rv_r\\
v_{2r}=v_r^2 - 2 (\text{Norm}(\epsilon))^r.
\end{gather*}

\begin{theorem}  Let $n\equiv 2\pmod 4$. Write $n^2+16 = y^2 d$ and assume
$$
\epsilon = \frac{n+y\sqrt{d}}{4}$$
is the fundamental unit of $k_n$.  Let $r_0$ be the smallest odd positive index for which $u_r$ is divisible by $d$ and suppose 
there is a prime $p\equiv 1\pmod 4$ dividing $v_{r_0}$ such that $du_{r_0}$ is a quadratic nonresidue mod $p$. 
Then there is no positive integer $m\ne n$ with $K_m=K_n$.
\end{theorem}
\begin{proof} As before, write $k_n=\mathbb Q(\sqrt{d})$ with $d$ squarefree. Then $n^2+16 = dy^2$ for some even integers $n, y$, and
$$
\epsilon = \frac{n+y\sqrt{d}}{4}.
$$ 
Suppose $K_m=K_n$.
From the proof of parts (a) and (b) of Theorem \ref{mainthm} in Section 3, we know that $m^2+16=dx^2$ for some $x$, and we have 
$$
 \frac{\sqrt{d}}{\epsilon-\overline{\epsilon}}(x/2) = u_r
$$ for some odd integer $r$. Moreover, $u_r=u_1 u_r=d h^2$ for some integer $h$, because Cohn's result says that $u_1 u_r$ cannot be a square.
Therefore, $d$ divides $u_r$. Let $g=\gcd(r_0, r)$. Since  $u_g = \gcd(u_{r_0}, u_r)$ and  $r_0$ is the smallest positive index with $d$ dividing $u_{r_0}$, 
we cannot have $g<r_0$, so $g=r_0$ and $r_0\mid r$, which implies that $r_0$ is odd. Therefore, we may write $r=(1+ 2\ell) r_0$ for some $\ell \ge 0$. 

\begin{lemma}
Let $r_1, r_2$ be positive integers with $r_1\mid r_2$ and $r_2/r_1$ even. If $p$ is a prime with $p\mid v_{r_1}$, then $v_{r_2}\equiv \pm 2\pmod p$.
\end{lemma}
\begin{proof}
Write $r_2/r_1=2^e\ell$ for some $e>0$ and some odd $\ell>0$. Since 
$$
v_{r_1}=\epsilon^{r_1} + \overline{\epsilon}^{r_1} \text{ divides } v_{r_1\ell}=\epsilon^{r_1\ell} + \overline{\epsilon}^{r_1\ell},
$$
we have $p\mid v_{r_1\ell}$.  Therefore, $v_{2r_1\ell} = v_{r_1\ell}^2-2 \equiv -2\pmod p$.
This implies $v_{4r_1\ell} = v_{2r_1\ell}^2-2 \equiv 2\pmod p$. Continuing in this way, we obtain the desired result.
\end{proof}

Returning to the proof of the theorem and working mod $p$, we have
\begin{align*}
2u_r &\equiv 2u_{r_0+2\ell r_0} \equiv u_{r_0} v_{2r_0\ell} + u_{2r_0\ell}v_{r_0}\\
& \equiv  u_{r_0} v_{2r_0\ell}  \equiv \pm 2  u_{r_0} \pmod p
\end{align*}
by the lemma.
Therefore,
$$
\left(\frac{du_r}{p}\right) = \left(\frac{\pm d u_{r_0}}{p}\right) = \left(\frac{\pm 1}{p}\right)\left(\frac{d u_{r_0}}{p}\right) = -1
$$
since $p\equiv 1\pmod 4$ and $d u_{r_0}$ is a quadratic non-residue mod $p$ by assumption.
Therefore, $d u_r$ cannot be a square. This completes the proof.
\end{proof}

{\bf Examples.}  $\boldsymbol{n=6:}$ The fundamental unit of $k_6$ is
$$\epsilon = \frac{6+\sqrt{52}}{4}=\frac{3+\sqrt{13}}{2}.$$
We have $u_{13}=13\times 118717$, so $r_0=13$. Since $v_{13}= 5564523\equiv 0 \pmod {53}$,  we take $p=53$. Since $13 u_{13}$ is a quadratic nonresidue mod $53$, the theorem applies and we find that there is no $m\ne 6$ with $K_m=K_6$.

$\boldsymbol{n=10:}$ The fundamental unit is $\epsilon = (5+\sqrt{29})/2$. We have $r_0=29$ and $p=64233493$. There is no other $m$ with $K_m=K_{10}$.

$\boldsymbol{n=14:}$ The fundamental unit is $\epsilon = (7+\sqrt{53})/2$. We have $r_0=53$ and $p=408359633417260832077$. There is no other $m$ with $K_m=K_{14}$.

These are the only examples we have found where the theorem applies. There are probably more, but finding a suitable prime $p$ requires 
finding factors of $v_{r_0}$, which is very large.

{\bf Example.} Consider $K_4$. The fundamental unit $\epsilon = 1+\sqrt{2}$ of the quadratic subfield $k_4=\mathbb Q(\sqrt{2})$ has norm $-1$ but has even trace, 
so Theorem \ref{mainthm} does not apply. We know that $u_1=1^2$ and $u_7=13^2$. Therefore, $u_1u_7$ is a square, which corresponds to the equality $K_4=K_{956}$. Peth\H o \cite{petho}
showed that $u_1$ and $u_7$ are the only squares in this sequence. Therefore, if there if there is another $\ell$ with $K_{\ell}=K_4=K_{956}$, then we obtain
an odd index $s$ with $u_1 u_s $ equaling twice a square (that is, $d$ times a square). But the $u_j$ satisfy $u_{j+2}= 2u_{j+1}+u_j$, so all the terms with odd indices are odd.
Therefore, there is no other $\ell$ with $K_{\ell}=K_4=K_{956}$.

\section{Elliptic Curves}\label{S:6}

\begin{proposition} Let $t$ be a positive integer. Consider the sequence defined by $u_0=0, \quad u_1=1, \quad u_{j+2}=t u_{j+1} + u_j$.
\newline
(a) There is an injection from the set of $u_j$ with odd $j$  in the sequence that are squares to the set of integer points $(x,y)$ on the curve 
$$C_1: \, y^2=(t^2+4)x^4-4$$ with $x, y\ge 0$. 
The point $(x,y)$ corresponds to $u_j=x^2$.
\newline
(b)  Write $t^2+4 = d z^2$ for an integer $z$ and squarefree $d$. There is an injection from the set of $u_j$ with odd $j$  in the sequence such that  $u_j/d$ is the square of an integer
to the set of integer points $(x,y)$ on the curve 
$$C_2:\, y^2=d^2(t^2+4)x^4-4$$ with $x, y\ge 0$. The point $(x, y)$ corresponds to $u_j=dx^2$.\newline
(c) If $\omega=(t+\sqrt{t^2+4})/2$ is the fundamental unit of $\mathbb Q(\sqrt{t^2+4})$, then the injections of parts (a) and (b) are bijections.
\end{proposition}

\begin{proof}
Let $\overline{\omega}=(t-\sqrt{t^2+4})/2$ and $u_j=(\omega^j-\overline{\omega}^j)/(\omega-\overline{\omega})$. Let $v_j=\omega^j+\overline{\omega}^j$.
Then 
$$
(t^2+4)u_j^2+4(\omega\overline{\omega})^j=v_j^2.$$
If $j$ is odd and $u_j$ is a square, the point $(\sqrt{u_j}, v_j)$ lies on the curve $C_1$.
Similarly, if $j$ is odd and $u_j/d$ is a square, the point $(\sqrt{u_j/d}, v_j)$ lies on the curve $C_2$.

Conversely, suppose $(x, y)$ is a point on the curve $C_1$ with $y\ge 0$. Let 
$$
\alpha=\frac{y+x^2\sqrt{t^2+4}}{2}.
$$
If $t$ is odd, $x$ and $y$ have the same parity, so $\alpha$ is an algebraic integer.  If $t$ is even, $y$ is even and $(1/2)\sqrt{t^2+4}$ is an algebraic integer, so again $\alpha$ is an algebraic integer. The norm of $\alpha$ is $-1$, so $\alpha=\omega^j$ for some odd $j$. It follows that $x^2=u_j$, so the map in part (a) is surjective. The case of $C_2$
follows similarly.
\end{proof}

\begin{proposition}\label{ellprop} Let $t$ and $d$ be nonzero integers. Define the elliptic curves
$$
E_1: \, Y_1^2= X_1^3-4(t^2+4)X_1, \qquad E_2:\, Y_2^2=X_2^3 - 4d^2(t^2+4)X_2.
$$
There are degree 2 rational maps $\phi_i: C_i\to E_i$ given by
\begin{gather*}
\phi_1(x, y) = \left((t^2+4)x^2, (t^2+4)xy\right), \\  \phi_2(x, y) = \left(d^2(t^2+4)x^2, \,  d^2(t^2+4)xy\right)
\end{gather*}
\end{proposition}
\begin{proof} For $\phi_2$, we have
\begin{gather*}
Y_2^2= d^4(t^2+4)^2x^2y^2= d^4(t^2+4)^2x^2\left(d^2(t^2+4)x^4-4\right)\\
=\left(d^2(t^2+4)x^2\right)^3-4d^2(t^2+4)\left(d^2(t^2+4)x^2\right) = X_2^3 - 4d^2(t^2+4)X_2.
\end{gather*}
The case $d=1$ is the calculation for $\phi_1$.\end{proof}

Note that $\phi_1$ and $\phi_2$ take integral points on $C_1$ and $C_2$ to integral points on $E_1$ and $E_2$, respectively.
The point $(1,t)$ on $C_1$ maps to $(t^2+4, t(t^2+4))$ on $E_1$. There is another point $(-4, 4t)$ on $E_1$, and a quick calculation shows that
$$
(t^2+4, t(t^2+4)) + (0,0)=(-4, 4t).
$$
\begin{proposition} The only non-trivial torsion point in $E_2(\mathbb Q)$ is $(0,0)$ (the case $d=1$ gives the result for $E_1$) . \end{proposition}
\begin{proof} 
Since $4d^2(t^2+4)$ cannot be a square,   Theorem 5.2 of Knapp \cite{knapp} implies that the torsion subgroup has order 2. \end{proof}

Note that $(0,0) \not\in \phi_i(C_i(\mathbb Q))$. Therefore, a (finite) rational point on $C_i$ gives a point of infinite order in the Mordell-Weil group of $E_i$.
Calculation of the root number $w$ of $E_1$ using the formulas of Birch and Stephens \cite{birch_stephens1} yields the following:
$$
w=\begin{cases} -1 \text{ if } t \not\equiv 0\pmod 8, \\  +1 \text{ if } t \equiv 0\pmod 8. \end{cases}
$$
The formula provided by Birch and Stephens applies to curves defined by an equation of the form $y^2 = x^3-Dx$ 
with $D \in \Z  \smallsetminus 4\Z$ and fourth-power free.  This provides no obstacle to our use of their formula, as it is 
straightforward to find a curve isogenous to E that has the requisite form.

For example, suppose that $t \equiv 0 \pmod{8}$.  Then $t^2+4 \equiv 0 \pmod{4}$ and so we may rewrite the equation defining $E_1$ 
as $Y_1^2 = X_1^3-16((t/2)^2+1)X_1$.  Write $(t/2)^2+1 = rs^4$ with $r > 0$ and fourth-power free.  Since $t \equiv 0 \pmod{8}$, we conclude that $1 \equiv (t/2)^2+1 = rs^4 \equiv r  \pmod{16}$.  In particular we note that $r \not\equiv 0 \pmod{4}$.  The curve $E_1$ is isomorphic to the 
curve \[
E:\, Y^2 = X^3-rX,
\] via $(X_1,\,Y_1) \mapsto (X_1/(2s)^2,\,Y_1/(2s)^3)$.  The formula in \cite{birch_stephens1} says $w_\infty(E) = \sgn(-r)=-1$.  
Since $r \equiv 1 \pmod{16}$, we have that $w_2(E(\Q))=-1$.  We note that $r$ has no prime divisor $p \equiv 3 \pmod{4}$, since the existence of such a prime $p$ would imply that $(t/2)^2 \equiv -1 \pmod{p}$ and hence that $-1$ is a quadratic 
residue modulo $p$, an impossibility.  Therefore, for each odd prime $p$ divisor of $r$, the local root number $w_p(E) = +1$.  Hence the global root number
\[
w(E) = w_\infty w_2 \prod_{p^2 \mid\mid r} w_p = +1.
\]

The case when $t \not\equiv 0 \pmod{8}$ can be handled similarly.

If it could be shown for some $t$ that the only integral points on $E_1$ are $(t^2+4, t(t^2+4))$
and $(-4, 4t)$, or if we could find all integral points on $E_2$, then we would obtain information on the field isomorphism problem.

In Proposition \ref{ellprop}, we could have chosen to have birational maps to the isogenous curves  $Y_1^2= X_1^3+(t^2+4)X_1$ and $ Y_2^2=X_2^3 +d^2(t^2+4)X_2$, but then integral points on $C_i$ would not necessarily map to integral points on $E_i$.  In the cases we are interested in, $t^2+4=x_0^2 d$ for some $x_0$. There is an isomorphism
\begin{gather*}
\psi: E_2:\, Y_2^2=X_2^3 - 4d^2(t^2+4)X_2 \longrightarrow E_3:\, Y_3^2=X_3^3 - 4(t^2+4)^3X_3\\
 (X_2, Y_2) \longmapsto (x_0^2 X_2, x_0^3 Y_2).
\end{gather*}
This has the advantage that $E_3$ lies in an explicit one-parameter family.

Computations indicate that the Mordell-Weil rank of $E_3$ is usually 0 or 1. 
When it is $0$, we have that $C_2$ has no integral points.  In a manner similar to that discussed above, application of the formulas of Birch and Stephens \cite{birch_stephens1} to (curves isogenous to) $E_3$ yields that the 
global root number $w(E_3) = (-1)^t$.

While doing computations with these curves, we noticed the following:
when $t$ is even, the rank of $E_3$ is usually 0. However, if $t$ is even and the rank of $E_3$ is 2, then the rank of $E_1$ is typically greater than or equal to 2. 
This holds for the 46 cases with $0<  t\le 3000$ with the one exception of $t=2990$, where the rank of $E_3$ is 2 and the rank of $E_1$ is 1.   We do not currently have an explanation for either the general phenomenon when the rank of $E_3$ is $2$ or for the exception when $t=2990$.

\bibliographystyle{amsplain}
\bibliography{refs}

\end{document}